\documentclass[12pt,reqno,a4paper]{amsart}

\usepackage[colorlinks,plainpages,hypertexnames=false,plainpages=false]{hyperref}
\hypersetup{urlcolor=blue, citecolor=blue, linkcolor=blue}

\usepackage[utf8]{inputenc}
\usepackage{amsmath}
\usepackage{amsthm}
\usepackage{amssymb}
\usepackage{amsfonts}
\usepackage{enumerate}
\usepackage{stmaryrd}
\usepackage{enumitem}
\usepackage{centernot}
\usepackage{todonotes}
\usepackage{mathrsfs}
\usepackage{listings}
\usepackage{multicol}
\usepackage[capitalise]{cleveref} %
\usepackage[noend]{algorithmic} %

\usepackage{mathdots}

\usepackage{array}
\newcolumntype{L}[1]{>{\raggedright\let\newline\\\arraybackslash\hspace{0pt}}m{#1}}
\newcolumntype{C}[1]{>{\centering\let\newline\\\arraybackslash\hspace{0pt}}m{#1}}
\newcolumntype{R}[1]{>{\raggedleft\let\newline\\\arraybackslash\hspace{0pt}}m{#1}}

\usepackage{tikz}
\usetikzlibrary{arrows}
\usetikzlibrary{decorations.pathreplacing}
\usetikzlibrary{matrix}
\usetikzlibrary{calc}
\usetikzlibrary{shapes}
\usetikzlibrary{patterns}
\usetikzlibrary{fit,backgrounds,scopes}

\newtheoremstyle{theoremstyle}
{10pt}      %
{5pt}       %
{\itshape}  %
{}          %
{\bfseries} %
{}         %
{ }      %
{}          %

\newtheoremstyle{algorithmstyle}
{10pt}      %
{5pt}       %
{}  %
{}          %
{\bfseries} %
{}         %
{ }      %
{}          %

\newtheoremstyle{examplestyle}
{10pt}      %
{5pt}       %
{}          %
{}          %
{\bfseries} %
{}         %
{ }      %
{}          %

\theoremstyle{theoremstyle}
\newtheorem{theorem}{Theorem}[section]
\newtheorem*{theorem*}{Theorem}
\newtheorem{lemma}[theorem]{Lemma}
\newtheorem{proposition}[theorem]{Proposition}
\newtheorem*{proposition*}{Proposition}

\newtheorem*{corollary*}{Corollary}

\theoremstyle{examplestyle}
\newtheorem{example}[theorem]{Example}
\newtheorem{definition}[theorem]{Definition}

\newtheorem{definition*}{Definition}

\newtheorem{remark}[theorem]{Remark}
\newtheorem{remark*}{Remark}
\newtheorem{convention}[theorem]{Convention}

\theoremstyle{algorithmstyle}
\newtheorem{algorithm}[theorem]{Algorithm}

\newcommand{\bluetriangleright}{\footnotesize $\textcolor{blue!90}{\blacktriangleright}$}

\newcommand{\NN }{\mathbb{N}}
\newcommand{\CC }{\mathbb{C}}
\newcommand{\RR }{\mathbb{R}}

\newcommand{\ZZ }{\mathbb{Z}}
\newcommand{\FF }{\mathbb{F}}

\newcommand{\suchthat}{\;\ifnum\currentgrouptype=16 \middle\fi|\;}

\newcommand{\bigslant}[2]{{\raisebox{.2em}{$#1$}\left/\raisebox{-.2em}{$#2$}\right.}}

\newcommand{\Cplusplus}{C\nolinebreak\hspace{-.05em}\raisebox{.4ex}{\tiny\bf +}\nolinebreak\hspace{-.10em}\raisebox{.4ex}{\tiny\bf +}}
\newcommand{\AlgComment}[2]{\hspace{#1}\textit{\footnotesize \textcolor{black!70}{// #2}}}
\newcommand{\LstComment}[2]{\hspace{#1}\textit{\scriptsize \textcolor{black!70}{// #2}}}

\DeclareMathOperator{\Star}{Star}

\DeclareMathOperator{\Div}{Div}

\DeclareMathOperator{\initial}{in}

\DeclareMathOperator{\id}{id}

\DeclareMathOperator{\Relint}{Relint}

\DeclareMathOperator{\Trop}{Trop}

\DeclareMathOperator{\Lin}{Span}

\DeclareMathOperator{\stable}{st}

\renewcommand{\P}{\mathbb{P}}
\renewcommand{\O}{\mathcal{O}}
\DeclareMathOperator{\Cox}{Cox}
\DeclareMathOperator{\Bl}{Bl}

\newcommand\restr[2]{{\left.\kern-\nulldelimiterspace #1 \right|_{#2}}}

\makeatletter
\newcommand{\customlabel}[2]{%
   \protected@write \@auxout {}{\string \newlabel {#1}{{#2}{\thepage}{#2}{#1}{}} }%
   \hypertarget{#1}{#2}
}
\makeatother

\usepackage{etoolbox}
\AtBeginEnvironment{pmatrix}{\setlength{\arraycolsep}{2pt}}

\makeatletter
\newcommand{\subalign}[1]{%
  \vcenter{%
    \Let@ \restore@math@cr \default@tag
    \baselineskip\fontdimen10 \scriptfont\tw@
    \advance\baselineskip\fontdimen12 \scriptfont\tw@
    \lineskip\thr@@\fontdimen8 \scriptfont\thr@@
    \lineskiplimit\lineskip
    \ialign{\hfil$\m@th\scriptstyle##$&$\m@th\scriptstyle{}##$\crcr
      #1\crcr
    }%
  }
}
\makeatother

\begin{document}

\title[Detecting tropical defects of polynomial equations]{Detecting tropical defects of \\ polynomial equations}
\author{Paul G\"orlach}
\address{Max Planck Institute for Mathematics in the Sciences\\
  Inselstra{\ss}e 22\\
  04103 Leipzig\\
  Germany
}
\email{paul.goerlach@mis.mpg.de}
\urladdr{https://personal-homepages.mis.mpg.de/goerlach}
\author{Yue Ren}
\address{Max Planck Institute for Mathematics in the Sciences\\
  Inselstra{\ss}e 22\\
  04103 Leipzig\\
  Germany
}
\email{yue.ren@mis.mpg.de}
\urladdr{https://yueren.de}

\author{Jeff Sommars}
\address{University of Illinois Chicago\\
  322 Science and Engineering Offices\\
  851 S. Morgan Street\\
  Chicago, IL 60607\\
  USA
}
\email{sommars1@uic.edu}
\urladdr{https://homepages.math.uic.edu/\textasciitilde sommars}

\subjclass[2010]{14T04, 13P10, 68W30}

\date{\today}

\keywords{tropical geometry, tropical basis, computer algebra.}

\begin{abstract}
  We introduce the notion of tropical defects, certificates that a system of polynomial equations is not a tropical basis, and provide two algorithms for finding them
in
  affine spaces of complementary dimension to the zero set. We use these techniques to solve open problems regarding del Pezzo surfaces of degree~3 and realizability of valuated gaussoids on~$4$ elements. %
\end{abstract}

\maketitle

\section{Introduction}

The \emph{tropical variety} $\Trop(I)$ of a polynomial ideal $I$ is the image of its algebraic variety under component-wise valuation. Tropical varieties are commonly described as combinatorial shadows of their algebraic counterparts and arise naturally in many applications throughout mathematics and beyond. Inside mathematics for example, they enable new insights into important invariants in algebraic geometry \cite{Mikhalkin05} or the complexity of central algorithms in linear optimization \cite{ABGJ18}. Outside mathematics they arise as spaces of phylogenetic trees in biology \cite{SpeyerSturmfels04,PachterSturmfels05}, loci of indifference prizes in economics \cite{BaldwinKlemperer18,TranYu15} or in the proof of the finiteness of central configurations in the $4,5$-body problem in physics \cite{HM06,HJ11}.

As the image of an algebraic variety, a tropical variety equals the intersection of all tropical hypersurfaces of the polynomials inside the ideal. A natural question in this context is whether this equality already holds for a given finite generating set $F\subseteq I$, i.e.,
\begin{equation}
  \tag{$\ast$}\label{eq:tropicalBasis}
  \Trop(I)=\bigcap_{f\in I} \Trop(f)\stackrel{?}{=}\bigcap_{f\in F}\Trop(f)=:\Trop(F).
\end{equation}
We call $\Trop(F)$ a \emph{tropical prevariety} and, if equality holds, $F$ a \emph{tropical basis}. This question is important for two main reasons. On the one hand, tropical prevarieties can provide upper dimension bounds where Gr\"obner bases are infeasible to compute, see \cite{HM06,HJ11}, and a tropical basis implies that this bound is actually sharp. On the other hand, the difference between a tropical variety and prevariety can be interesting in and of itself, e.g., tropical matrices of Kapranov rank~$r$ versus tropical matrices of tropical rank~$r$ \cite{DSS05}, tropical Grassmannians versus their Dressians \cite{HJS14}, or other realizability loci of combinatorial objects such as $\Delta$-matroids \cite{Rincon12} or gaussoids \cite{BDKS17}.

Nevertheless, checking the equality in \eqref{eq:tropicalBasis} is a computationally highly challenging task. Current algorithms for computing tropical varieties require a Gr\"obner basis for each maximal Gr\"obner polyehdron, of which there can be many even for tropicalization of linear spaces \cite{JoswigSchroeter18}. Additionally, it is known that deciding the equality in \eqref{eq:tropicalBasis} is co-NP-hard, as is merely deciding whether $\Trop(F)$ is connected \cite{Theobald06}.   %

In practice, testing the equality in \eqref{eq:tropicalBasis} can fail for multiple reasons:
\begin{enumerate}[leftmargin=12mm]
\item[\customlabel{P1}{(P1)}] Computing $\Trop(F)$ might not be possible due to its size or due to the number of intersections necessary to compute it.
\item[\customlabel{P2}{(P2)}] Computing $\Trop(I)$ might not be feasible due to its size or due to problematic Gr\"obner cones in $\Trop(I)$ whose Gr\"obner bases are too hard to compute.
\end{enumerate}

In this article, we introduce the notion of \emph{tropical defects}, certificates for generating sets which are not tropical bases, and propose two randomized algorithms for computing tropical defects around affine subspaces of complementary dimension.
An independent verification of these certificates will require a single Gröbner basis computation.

The basic idea is simple, relying on some recent results on (stable) intersections of tropical varieties \cite{OP13,JY16}: to reduce the complexity of the computations, we (stably) intersect both sides of Equation \eqref{eq:tropicalBasis} with a random affine space of complementary dimension, and look for differences between the tropical variety and prevariety around it. Under certain genericity assumptions, this yields a zero-dimensional tropical variety on the left, which is not only simpler to compute than its positive-dimensional counterparts, but also implies that the tropical prevariety computation on the right can be aborted if a positive-dimensional polyhedron is found. Therefore, our algorithm operates within the realm where \ref{P1} and \ref{P2} are infeasible, but the following key computational ingredients are not:
\begin{itemize}[leftmargin=12mm]
\item[\customlabel{K1}{(K1)}] computation of zero-dimensional tropical varieties in \textsc{Singular} \cite{singular,HofmannRen16},
\item[\customlabel{K2}{(K2)}] computation of zero-dimensional tropical prevarieties in \textsc{DynamicPrevariety} \cite{JSV17}.
\end{itemize}

To a degree, our approach for finding tropical defects is related to the approach for studying tropical bases in \cite{HT09,HT12}. In \cite{HT09,HT12}, the authors consider preimages of projections to $\RR^{d+1}$, where $d:=\dim \Trop(I)$. Our hyperplanes are generally given as preimages of points under a projection to $\RR^d$, but can also be regarded as preimages of lines under a projection to $\RR^{d+1}$. Hence our approach can be seen as a relaxation where instead of considering the preimage of the entire projection to $\RR^{d+1}$ we only consider the parts of the projection which meet a fixed line.

In Sections~\ref{sec:DelPezzo} and~\ref{sec:Gaussoids}, we present two tropical defects found using out algorithm, disproving Conjecture 5.3 in \cite{RSS16} and Conjecture~8.4 in \cite{BDKS17}. Note that the tropical defects were postprocessed for the ease of reproduction, see Remark~\ref{rem:singletonDefects}.
Code and auxiliary materials for this article are available at \href{https://www.software.mis.mpg.de}{software.mis.mpg.de}. More information on gaussoids can be found at \href{https://www.gaussoids.de}{gaussoids.de}.

\subsection*{Acknowledgements}
The authors would like to thank Bernd Sturmfels for his helpful comments and suggestions. All authors were partially supported by the Institut Mittag-Leffler during the research program ``Tropical Geometry, Amoebas and Polytopes''. The authors would like to thank the institute for its hospitality.

\section{Tropical defects}
In this section, we introduce the notion of tropical defects for generating sets of polynomial ideals, and two algorithms to find them around generic affine spaces $L=\Trop(H)$ of complementary dimension. To be precise, Algorithm~\ref{alg:strongGenericity} requires a generic tropicalization $L$, whereas Algorithm~\ref{alg:weakGenericity} merely requires a generic realization $H$.

We begin by briefly recalling some basic notions of tropical geometry that are of immediate relevance to us. Our notation coincides with that of \cite{MS15}, to which we refer for a more in-depth introduction of the subject. %

\begin{convention}
  For the remainder of this article, fix an algebraically closed field~$K$ with valuation $\nu\colon K^\ast \rightarrow \RR$ and residue field $\mathfrak K$ with trivial valuation. Since $K$ is algebraically closed, there is a group homomorphism $\mu:\nu(K^\ast)\rightarrow K^\ast$ such that $\nu\circ\mu=\id_{\nu(K^\ast)}$, and we abbreviate $t^{\lambda}:=\mu(\lambda)$ for $\lambda\in\nu(K^\ast)$. Moreover, we fix a multivariate (Laurent) polynomial ring $K[x^{\pm 1}]:=K[x_1^{\pm 1},\ldots,x_n^{\pm 1}]$.
\end{convention}

\begin{definition}[Initial forms, initial ideals]
  Given a polynomial $f\in K[x^{\pm 1}]$, say $f=\sum_{\alpha\in\ZZ^n} c_{\alpha}\cdot x^{\alpha}$, its \emph{initial form} with respect to a weight vector $w\in\RR^n$ is
  \begin{align*}
    \initial_w(f) {}:={}& \textstyle\sum_{w\cdot \alpha + \nu(c_{\alpha}) \text{ min.}} \overline{t^{-\nu(c_{\alpha})}c_{\alpha}}\cdot x^{\alpha} &&\in\mathfrak K[x^{\pm 1}].\\
    \intertext{For a finite set $F \subseteq K[x^{\pm 1}]$ and an ideal $I\unlhd K[x^{\pm 1}]$, we denote}
    \initial_w(F) \!:={}& \{\initial_w(g) \mid g \in F\} &&\subseteq \mathfrak K[x^{\pm 1}],\\
    \initial_w(I) {}:={}& \langle \initial_w(g)\mid g\in I\rangle&&\unlhd\mathfrak K[x^{\pm 1}].\\
    \intertext{Moreover, the \emph{Gr\"obner polyhedron} of $f$, of $I$ or of a finite set $F\subseteq K[x^{\pm 1}]$ around $w$ is defined as}
    C_w(f) :=& \overline{\{v\in\RR^n\mid \initial_w(f)=\initial_v(f) \}}&&\subseteq \RR^n,\\
    C_w(I) :=& \overline{\{v\in\RR^n\mid \initial_w(f)=\initial_v(f) \text{ for all } f\in I\}}&&\subseteq \RR^n,\\
    C_w(F) :=& \overline{\{v\in\RR^n\mid \initial_w(f)=\initial_v(f) \text{ for all } f\in F\}}\hspace{-1.75cm}&&\subseteq \RR^n.
  \end{align*}
  Note that both $C_w(f)$ and $C_w(F)$ are in fact convex polyhedra, while $C_w(I)$ is only guaranteed to be a convex polyhedron if $I$ is homogeneous.
\end{definition}

\begin{definition}[Tropical variety, tropical prevariety]
  Given a polynomial $f\in K[x^{\pm 1}]$, an ideal $I\unlhd K[x^{\pm 1}]$ and a finite set $F\subseteq K[x^{\pm 1}]$, the \emph{tropical varieties} of $f$ and $I$ and the \emph{tropical prevariety} of $F$ are defined to be
  \begin{align*}
    \Trop(f) &:= \{w \in \RR^n \mid \initial_w(f) \text{ is not a monomial}\},\\
    \Trop(I) &:= \{w \in \RR^n \mid \initial_w(f) \text{ is not a monomial for all } f\in I\},\\
  \Trop(F) &:= \{w \in \RR^n \mid \initial_w(f) \text{ is not a monomial for all } f\in F\}.
  \end{align*}
  We call a finite generating set $F\subseteq I$ a \emph{tropical basis} if
  \[ \Trop(F) = \Trop(I). \]
  Note that $\Trop(f)$, $\Trop(I)$ and $\Trop(F)$ are supports of polyhedral complexes. For both $\Trop(f)$ and $\Trop(F)$ these polyhedral complexes can be chosen to be a collection of Gr\"obner polyhedra, and, if $I$ is homogeneous, so can $\Trop(I)$.

  Let $T\subseteq \RR^n$ be the support of a polyhedral complex $\Sigma$. Recall that the \emph{star} of $T$ around a point $w\in\RR^{n}$ is given by
  \[ \Star_w T := \{v\in\RR^n\mid w+\varepsilon\cdot v\in T \text{ for }\varepsilon>0 \text{ sufficiently small}\} \]
  and that the \emph{stable intersection} of $T$ with respect to an affine subspace $H\subseteq \RR^n$ is defined to be
  \[ T\cap_{\stable} H := \bigcup_{\substack{\sigma\in\Sigma\\ \dim(\sigma + H)=n}} \sigma \cap H. \]
\end{definition}

\begin{example}\label{ex:tropicalBases}\
  Let $K=\CC\{\!\{t\}\!\}$ be the field of complex Puiseux series and consider the ideal $I\unlhd K[x^{\pm 1},y^{\pm 1}]$ which can be generated by either one of the following two generating sets:
  \[ I:=\langle \underbrace{\textcolor{blue}{x+y+1},\textcolor{orange}{x+t^{-1}y+2}}_{=:F_1}\rangle = \langle \underbrace{\textcolor{blue}{x+y+1},\textcolor{orange}{(t^{-1}-1)y+1}}_{=:F_2}\rangle \]
  Figure~\ref{fig:tropicalBases} compares the tropical prevarieties of both $F_1$ and $F_2$ with the tropical variety of $I$, showing that $F_2$ is a tropical basis while $F_1$ is not.

  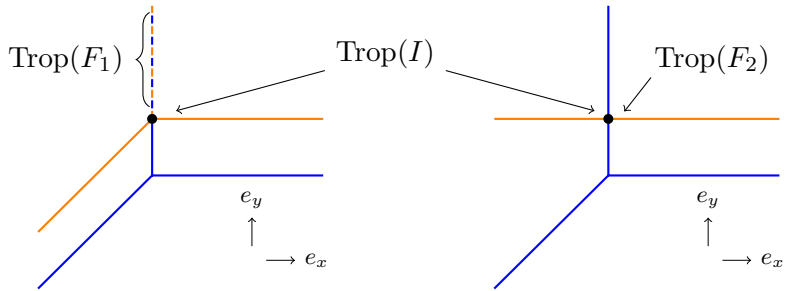
\begin{figure}[h]
    \centering
    \begin{tikzpicture}[every node/.style={font=\small}]
      \node[anchor=north] at (-3,0)
      {
        \begin{tikzpicture}[scale=1.5]
        \draw[->,very thin] (0.75,-0.875)++(0.25,0.125) -- ++(0.25,0)
        node[anchor=west,font=\scriptsize] {$e_x$};
        \draw[->,very thin] (0.75,-0.875)++(0.125,0.25) -- ++(0,0.25)
        node[anchor=south,font=\scriptsize] {$e_y$};

          \draw[thick,blue]
          (0,0) -- ++(1.5,0)
          (0,0) -- ++(0,0.5)
          (0,0) -- ++(-1,-1);
          \draw[thick,orange]
          (0,0.5) -- ++(1.5,0)
          (0,0.5) -- ++(-1,-1);
          \draw[thick,blue,dash pattern= on 3pt off 5pt,dash phase=4pt]
          (0,0.5) -- ++(0,1);
          \draw[thick,orange,dash pattern= on 3pt off 5pt]
          (0,0.5) -- ++(0,1);
          \fill
          (0,0.5) circle (1.25pt);
        \end{tikzpicture}
      };

      \draw[decorate,decoration={brace,amplitude=5pt,mirror},xshift=0pt,yshift=0pt]
      (-3.5,-0.25) -- (-3.5,-1.5) node[midway,anchor=east,xshift=-2mm] (TropF1) {$\Trop(F_1)$};
      \node[anchor=north] at (3,0)
      {
        \begin{tikzpicture}[scale=1.5]
          \draw[->,very thin] (0.75,-0.875)++(0.25,0.125) -- ++(0.25,0)
          node[anchor=west,font=\scriptsize] {$e_x$};
          \draw[->,very thin] (0.75,-0.875)++(0.125,0.25) -- ++(0,0.25)
          node[anchor=south,font=\scriptsize] {$e_y$};

          \draw[thick,blue]
          (0,0) -- ++(1.5,0)
          (0,0) -- ++(0,1.5)
          (0,0) -- ++(-1,-1);
          \draw[thick,orange]
          (-1,0.5) -- (1.5,0.5);
          \fill
          (0,0.5) circle (1.25pt);
        \end{tikzpicture}
      };

      \node[anchor=base,yshift=-9mm,xshift=1mm] (Trop) at (-0.5,0) {$\Trop(I)$};
      \draw[thin,->] (Trop) -- ++(-2.75,-0.75);
      \draw[thin,->] (Trop) -- ++(2.75,-0.75);
      \node[xshift=85mm] (TropF2) at (TropF1) {$\Trop(F_2)$};
      \draw[thin,->] (TropF2.south west)++(0.1,0.1) -- ++(-0.35,-0.45);
    \end{tikzpicture}\vspace{-6mm}
    \caption{A tropical non-basis and a tropical basis.}
    \label{fig:tropicalBases}
  \end{figure}
\end{example}

\vspace{-3mm}

For the following result, we refer to \cite{MS15}, where it is only shown for polynomial rings. However, the result extends directly to Laurent polynomial rings, since ${\initial_w(I \cap K[x]) \cdot K[x^{\pm 1}] = \initial_w(I)}$ for all $I \unlhd K[x^{\pm 1}]$.

\begin{lemma}[{\cite[Lemma~2.4.6 and Corollary~2.4.10]{MS15}}]\label{lem:nestedWeights}
  Given an element $f\in K[x^{\pm 1}]$ and a homogeneous ideal $I\unlhd K[x^{\pm 1}]$, we have for any weight vectors $w,v\in\RR^n$ and $\varepsilon>0$ sufficiently small:
  \[ \initial_v \initial_w(f) = \initial_{w+\varepsilon\cdot v}(f) \text{ and } \initial_v \initial_w(I) = \initial_{w+\varepsilon\cdot v}(I). \]
  In particular, for a finite set $F\subseteq K[x^{\pm 1}]$ or an ideal $I\unlhd K[x^{\pm 1}]$ this implies
  \[ \Trop(\initial_w F) = \Star_w \Trop(F) \text{ and } \Trop(\initial_w I) = \Star_w \Trop(I). \]
\end{lemma}

We will now introduce the notion of a tropical defect and two algorithms for finding them around affine spaces of complementary dimension. For the sake of simplicity, we restrict ourselves to affine spaces in direction of the last few coordinates, see Example~\ref{ex:strongGenericity} for general affine spaces.

\begin{definition}[Tropical defects]
  Let $I\unlhd K[x^{\pm 1}]$ be a polynomial ideal with finite generating set $F\subseteq I$. We call a finite tuple $\mathbf{w}:=(w_0,\dots,w_k) \in (\RR^n)^{k+1}$ a \emph{tropical defect} if for all $\varepsilon > 0$ sufficiently small we have
  \[ w_0 + \varepsilon w_1 + \dots + \varepsilon^{k} w_k \in \Trop(F)\setminus\Trop(I). \]
\end{definition}

\begin{example}\label{ex:tropicalDefects}
  For $I=\langle F_1\rangle$ from Example~\ref{ex:tropicalBases}, the tuple $(w,v)$ with $w:=(0,1)$ and $v:=(0,1)$ is a tropical defect, while the singleton $(w)$ is not. On the other hand, the singleton $(u)$ with $u:=(0,2)$ is a tropical defect, see Figure~\ref{fig:tropicalDefects}. \vspace{-7mm}
\end{example}

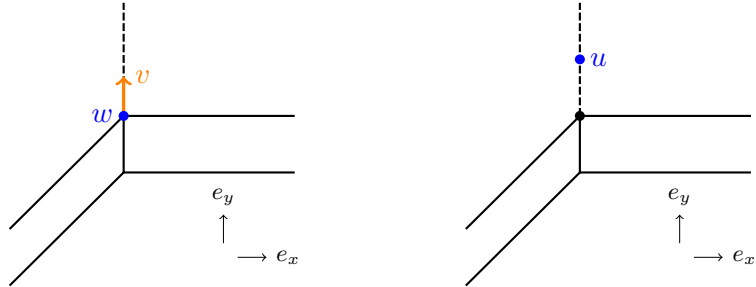
\begin{figure}[h]
  \centering
  \begin{tikzpicture}[every node/.style={font=\small}]
    \node[anchor=north] at (-3,0)
    {
      \begin{tikzpicture}[scale=1.5]
      \draw[->,very thin] (0.75,-0.875)++(0.25,0.125) -- ++(0.25,0)
      node[anchor=west,font=\scriptsize] {$e_x$};
      \draw[->,very thin] (0.75,-0.875)++(0.125,0.25) -- ++(0,0.25)
      node[anchor=south,font=\scriptsize] {$e_y$};

        \draw[thick]
        (0,0) -- ++(1.5,0)
        (0,0) -- ++(0,0.5)
        (0,0) -- ++(-1,-1);
        \draw[thick]
        (0,0.5) -- ++(1.5,0)
        (0,0.5) -- ++(-1,-1);
        \draw[thick,dash pattern= on 3pt off 5pt,dash phase=4pt]
        (0,0.5) -- ++(0,1);
        \draw[thick,dash pattern= on 3pt off 5pt]
        (0,0.5) -- ++(0,1);
        \draw[very thick,orange,->]
        (0,0.5) -- ++(0,0.35) node[anchor=west] {$v$};
        \fill[blue]
        (0,0.5) circle (1.25pt);
        \node[anchor=east,blue] at (0,0.5) {$w$};
      \end{tikzpicture}
    };

    \node[anchor=north] at (3,0)
    {
      \begin{tikzpicture}[scale=1.5]
      \draw[->,very thin] (0.75,-0.875)++(0.25,0.125) -- ++(0.25,0)
      node[anchor=west,font=\scriptsize] {$e_x$};
      \draw[->,very thin] (0.75,-0.875)++(0.125,0.25) -- ++(0,0.25)
      node[anchor=south,font=\scriptsize] {$e_y$};

        \draw[thick]
        (0,0) -- ++(1.5,0)
        (0,0) -- ++(0,0.5)
        (0,0) -- ++(-1,-1);
        \draw[thick]
        (0,0.5) -- ++(1.5,0)
        (0,0.5) -- ++(-1,-1);
        \draw[thick,dash pattern= on 3pt off 5pt,dash phase=4pt]
        (0,0.5) -- ++(0,1);
        \draw[thick,dash pattern= on 3pt off 5pt]
        (0,0.5) -- ++(0,1);
        \fill
        (0,0.5) circle (1.25pt);
        \fill[blue]
        (0,1) circle (1.25pt);
        \node[anchor=west,blue] at (0,1) {$u$};
      \end{tikzpicture}
    };
  \end{tikzpicture}\vspace{-6mm}
  \caption{Two tropical defects.}
  \label{fig:tropicalDefects}
\end{figure}
\vspace{-7mm}

\begin{remark}[Singleton tropical defects]\label{rem:singletonDefects}
  Note that any tropical defect $(w_0,\ldots,w_k)$ of a homogeneous ideal can be transformed into a singleton tropical defect $u$ through a single (tropical) Gr\"obner basis \cite{ChanMaclagan18} or standard basis computation \cite{MarkwigRenTropicalVarieties}:

  One can simulate the weight vector $w_\varepsilon:=w_0 + \varepsilon w_1 + \dots + \varepsilon^{k} w_k$ for $\varepsilon>0$ sufficiently small through a sequence of weights as in Lemma~\ref{lem:nestedWeights}. In particular, we can compute a Gr\"obner basis with respect to the sequence of weights, which gives us the inequalities and equations of the Gr\"obner cone $C_{w_\varepsilon}(I)$ by \cite[proof of Prop. 2.5.2]{MS15}. Any $u\in \Relint C_{w_\varepsilon}(I)$ is a singleton tropical defect.

  For the ease of verification, the tropical defects in Sections~\ref{sec:DelPezzo} and~\ref{sec:Gaussoids} have been transformed into singletons.
\end{remark}

Algorithm~\ref{alg:strongGenericity} checks for tropical defects around affine subspaces which satisfy a strong genericity assumption.

\begin{algorithm}[Testing for defects, strong genericity] \label{alg:strongGenericity}\
  \begin{algorithmic}[1]
    \REQUIRE{$(F,v)$, where
      \begin{enumerate}[leftmargin=3mm]
      \item $F\subseteq K[x^{\pm 1}]$, a finite generating set of a $d$-dimensional prime ideal $I\subseteq K[x^{\pm 1}]$, and assume w.l.o.g. that
        \begin{equation}
          \label{eq:projectionStrong}\tag{$\ast$}
          \pi(\Trop(I)) = \RR^d,
        \end{equation}
        where $\pi:\RR^n\rightarrow\RR^d$ denotes the projection onto the first $d$ coordinates.
      \item $v \in \RR^d$, describing an affine subspace $H:=\pi^{-1}(v) \subseteq\RR^n$ of complementary dimension $n-d$ such that the following \textit{strong genericity} assumption holds:
      \begin{equation} \label{eq:strongGenericity} \tag{SG}
        \Trop(I)\cap H = \Trop(I)\cap_{\stable} H.
      \end{equation}
    \end{enumerate}
    }
    \ENSURE{$(b,\mathbf{w})$, such that
      \begin{enumerate}[leftmargin=*]
      \item if b=\texttt{true}, then $\mathbf{w}$ is a tropical defect,
      \item if b=\texttt{false}, then $\Trop(F)\cap H = \Trop(I)\cap H$. (In this case, $\mathbf{w} := 0$.)
      \end{enumerate}
    }
    \STATE Set $F':=F\cup\{x_i-t^{v_i}\mid i=1,\dots,d \}$ and $I':= I+\langle x_i-t^{v_i}\mid i=1,\dots,d\rangle$.
    \STATE Compute the tropical prevariety $\Trop(F')$. %
    \IF{$\exists w \in \Trop(F')$ with $\dim C_w(F') > 0$}
    \STATE Pick $0\neq u\in \Lin(C_w(F')-w)$. \AlgComment{8mm}{where $C_w(F')-w:=\{v-w\mid v\in C_w(F')\}$}
    \RETURN{(\texttt{true}, $(w,u)$).} \label{algstep:SGposDim}
    \ENDIF

    \STATE Compute the tropical variety $\Trop(I')$. %
    \IF{$\exists w\in \Trop(F') \setminus \Trop(I')$}
    \RETURN{(\texttt{true}, $w$)} \label{algstep:SGstrongDefect}
    \ELSE
    \RETURN{(\texttt{false}, $0$)} \label{algstep:SGNoDefects}
    \ENDIF
  \end{algorithmic}
\end{algorithm}

\begin{proof}[Correctness of Algorithm~\ref{alg:strongGenericity}]
  Note that \eqref{eq:strongGenericity} implies that $\Trop(I) \cap H$ is at most zero-dimensional, since $H$ is of complementary dimension to $\Trop(I)$ and by \cite[Theorem 3.6.10]{MS15}, while \eqref{eq:projectionStrong} ensures that it is not empty. By \cite[Theorem~1.1]{OP13}, we therefore have
  \begin{align*}
  \Trop(I') &= \Trop(I+\langle x_i-t^{v_i}\mid i=1,\dots,d\rangle)  \\
  &= \Trop(I)\cap\Trop(\langle x_i-t^{v_i}\mid i=1,\dots,d\rangle) = \Trop(I) \cap H.
  \end{align*}

  If the algorithm terminates at Line~\ref{algstep:SGposDim}, then $C_w(F')$ is a positive-dimensional polyhedron contained in $\Trop(F') = \Trop(F) \cap H$, whereas $\Trop(I) \cap H$ consists of finitely many points.
  In particular, we have that $w+\varepsilon u\notin\Trop(I)$ for $\varepsilon>0$ sufficiently small.

  If the algorithm terminates at Line~\ref{algstep:SGstrongDefect}, then $w$ is a tropical defect since
  \[w \in \Trop(F')\setminus\Trop(I') = (\Trop(F) \cap H) \setminus (\Trop(I) \cap H) \subseteq \Trop(F)\setminus\Trop(I).\]

  Finally, should the algorithm terminate at Line~\ref{algstep:SGNoDefects}, then
  \[\Trop(F) \cap H = \Trop(F') = \Trop(I') = \Trop(I) \cap H. \qedhere\]
\end{proof}

\begin{example}\label{ex:strongGenericity}
  Consider the generating set $F$ of the following one-dimensional ideal:
  \[ I:= \langle \underbrace{(x+1)(y+1),(x-1)(y+1)}_{=:F} \rangle \subseteq \CC[x^{\pm1},y^{\pm1}], \]
  and let $\pi\colon\RR^{\{x,y\}}\rightarrow\RR^{\{x\}}$ denote the projection onto the $x$-coordinate. Figure~\ref{fig:exampleStrongGenericity} shows the tropical variety $\Trop(I)$ and the tropical prevariety $\Trop(F)$.

  Then for any $v\in\RR$ the affine line $H_v:=\pi^{-1}(v)$ satisfies \eqref{eq:strongGenericity}. Algorithm~\ref{alg:strongGenericity} yields a tropical defect if and only if $v=0$, in which %
  case
  it terminates at Line~\ref{algstep:SGposDim}. %

  \begin{figure}[h]
    \centering
    \begin{tikzpicture}[scale=1]
      \draw[->,very thin] (1.5,-2)++(0.5,0.25) -- ++(0.5,0)
      node[anchor=west,font=\scriptsize] {$e_x$};
      \draw[->,very thin] (1.5,-2)++(0.25,0.5) -- ++(0,0.5)
      node[anchor=south,font=\scriptsize] {$e_y$};

      \draw[very thin] (0,-2) -- (0,2);
      \draw[very thin] (1.5,2) -- (-2.5,-2);
      \draw[orange,thick] (-2.75,0) -- (2.75,0);
      \draw[orange,very thick] (0,-1.75) -- (0,1.75);
      \draw[blue,very thick] (-2.25,0) -- (2.25,0);
      \fill[blue] (0,0) circle (2pt);
      \node[blue,anchor=north west,font=\scriptsize] at (0,0) {$0$};
      \node[blue,anchor=south east,font=\footnotesize] at (2.25,0)
      {$\Trop(I)$};
      \node[orange,anchor=south east,font=\footnotesize] at (0,-1.75)
      {$\Trop(F)$};
      \node[anchor=north east,font=\footnotesize] at (0,2) {$H_0$};
      \node[anchor=north west,font=\footnotesize] at (1.25,2) {$L_{-1}$};
      \fill[blue] (-0.5,0) circle (1.25pt);
      \fill[orange] (0,0.5) circle (1.25pt);
    \end{tikzpicture}%
    \begin{tikzpicture}[scale=1]
      \draw[->,very thin] (1.5,-2)++(0.5,0.25) -- ++(0.5,0)
      node[anchor=west,font=\scriptsize] {$e_a$};
      \draw[->,very thin] (1.5,-2)++(0.25,0.5) -- ++(0,0.5)
      node[anchor=south,font=\scriptsize] {$e_b$};

      \draw[very thin] (2,-2) -- (-2,2);
      \draw[very thin] (-0.5,2) -- (-0.5,-2);
      \draw[orange,thick] (-2.75,0) -- (2.75,0);
      \draw[orange,very thick] (1.75,-1.75) -- (-1.75,1.75);
      \draw[blue,very thick] (-2.25,0) -- (2.25,0);
      \fill[blue] (0,0) circle (2pt);
      \node[blue,anchor=north,font=\scriptsize] at (0,0) {$0$};
      \node[blue,anchor=south east,font=\footnotesize] at (2.25,0) {$\Trop(\psi(I))$};
      \node[orange,fill=white,inner sep=1pt,anchor=south east,font=\footnotesize] at (1.25,-1.75) {$\Trop(\psi(F))$};
      \node[anchor=north east,font=\footnotesize] at (-1.8,2) {$(\psi^{\flat})^{-1}H_0$};
      \node[anchor=north west,font=\footnotesize] at (-0.5,2) {$(\psi^{\flat})^{-1}L_{-1}$};
      \fill[blue] (-0.5,0) circle (1.25pt);
      \fill[orange] (-0.5,0.5) circle (1.25pt);
    \end{tikzpicture}
    \vspace{-2mm}
    \caption{$\Trop(I)\subseteq\Trop(F)$ in Example~\ref{ex:strongGenericity}.}
    \label{fig:exampleStrongGenericity}
  \end{figure}
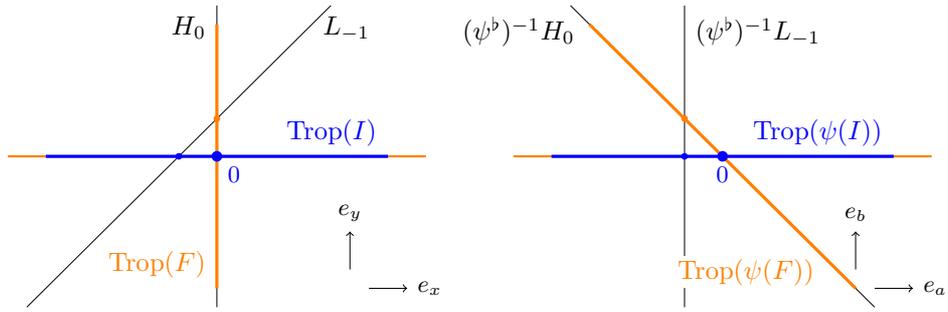

  We can also use arbitrary rational affine subspaces like $L_v := v\cdot e_x + \Lin(e_x+e_y)$ by applying a unimodular transformation $\psi$ on the ring of Laurent polynomials whose induced map $\psi^\flat$ on the weight space aligns $L_v$ with the coordinate axes:
  \begin{align*}
    &\psi\phantom{{}^\flat}\colon&K[x^{\pm 1},y^{\pm 1}] &\stackrel{\sim}{\longrightarrow} K[a^{\pm 1},b^{\pm 1}],& x&\mapsto ab,& \ y &\mapsto b,\\
    &\psi^\flat\colon &\RR^{\{x,y\}} &\stackrel{\sim}{\longleftarrow} \RR^{\{a,b\}},& e_x&\mapsfrom e_a,& e_x+e_y&\mapsfrom e_b.
  \end{align*}
  This transformation yields
  \begin{align*}
    \psi(F) &= \{(ab+1)(b+1),(ab-1)(b+1)\} \text{ and } \\
    (\psi^\flat)^{-1}(L_v) &= v\cdot e_a + \Lin(e_b) \subseteq \RR^{\{a,b\}},
  \end{align*}
  which always satisfies \eqref{eq:strongGenericity} and for which Algorithm~\ref{alg:strongGenericity} terminates at Line~\ref{algstep:SGstrongDefect} if and only if $v\neq 0$, as $\Trop(\psi(F))\cap (\psi^\flat)^{-1}(L_v)$ consists of two points of which only one belongs to the tropical variety $\Trop(\psi(I))$, see Figure~\ref{fig:exampleStrongGenericity}.
\end{example}

\begin{example}\label{ex:constructLinSp}
  Consider the generating set $F$ of the following one-dimensional ideal:
  \[ I:=\langle \underbrace{x+z+2, y+z+1}_{=:F} \rangle\unlhd \CC[x^{\pm 1},y^{\pm 1},z^{\pm 1}], \]
  and let $\pi\colon\RR^{\{x,y,z\}}\rightarrow\RR^{\{x\}}$ denote the projection onto the $x$-coordinate.
  Figure~\ref{fig:constructLinSp} shows $\Trop(I)$ as well as $\Trop(F)$. Consider the plane $H_v:=\pi^{-1}(v)$ for some $v\in\RR$.
  Note that while any $H_v$ with $v\neq 0$ satisfies \eqref{eq:strongGenericity}, only $H_v$ with $v>0$ yields a tropical defect in Algorithm~\ref{alg:strongGenericity}, Line~\ref{algstep:SGposDim}.
\end{example}
\begin{figure}[h]
  \centering
  \begin{tikzpicture}[x={(0.87cm,0.5cm)},y={(0cm,1cm)},z={(1cm,0cm)}]
    \fill[orange!50] (0,0,0) -- (2.25,0,0) -- (0,2.25,0);
    \node[orange!50!black,anchor=south west] at (1.05,1.05,0) {$\subseteq \Trop(F)$};
    \draw[very thick,->] (0,0,0) -- (2.5,0,0) node[anchor=west,font=\footnotesize] {$e_x$};
    \draw[very thick,->] (0,0,0) -- (0,2.5,0) node[anchor=west,font=\footnotesize] {$e_y$};
    \draw[very thick,->] (0,0,0) -- node[below,xshift=3mm] {$\Trop(I)$} (0,0,2.5) node[anchor=south,font=\footnotesize] {$e_z$};
    \draw[very thick,->] (0,0,0) -- (-1,-1,-1) node[anchor=north,font=\footnotesize] {$-e_x-e_y-e_z$};
    \fill (0,0,0) circle (2pt);

  \end{tikzpicture}\vspace{-3mm}
  \caption{$\Trop(I)\subseteq \Trop(F)$ from Example~\ref{ex:constructLinSp}}\label{fig:constructLinSp}.
\end{figure}
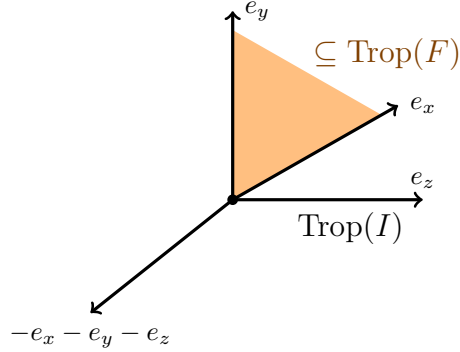

\begin{remark}[Strong genericity]
  In Algorithm~\ref{alg:strongGenericity}, the strong genericity assumption \eqref{eq:strongGenericity} is only required for the correctness of the output at Line~\ref{algstep:SGposDim}. If the algorithm does not terminate at Line~\ref{algstep:SGposDim}, then \eqref{eq:strongGenericity} must hold because $\Trop(F) \cap H = \Trop(F')$ is zero-dimensional, and hence so is $\Trop(I) \cap H \subseteq\Trop(F)\cap H$. This implies that for $\lambda_i \in K$ generic with $\nu(\lambda_i) = v_i$, we have
  \[\Trop(I) \cap H = \Trop(I+\langle x_i - \lambda_i \rangle) = \Trop(I) \cap_{\stable} H,\]
  where the first equality holds by \cite[Theorem~1.1]{OP13}, and the second equality holds by \cite[Theorem~3.6.1]{MS15}.

  One possibility to ascertain whether \eqref{eq:strongGenericity} holds upon termination at Line~\ref{algstep:SGposDim} is to compute the Gr\"obner polyhedron $C_w(I)$, if $I$ is homogeneous. However, that requires a tropical Gr\"obner basis or standard basis, and hence might not be viable for large examples.
\end{remark}

In practice, affine subspaces satisfying the strong genericity assumption induce several problems, see Remark~\ref{rem:strongVsWeak}. This is why we introduce Algorithm~\ref{alg:weakGenericity}, which relies on a weakened genericity assumption. Note that, compared to Algorithm~\ref{alg:strongGenericity}, Algorithm~\ref{alg:weakGenericity} requires the computation of $\Trop(\initial_w(F))$ for some $w\in\Trop(F)\cap H$ at Line~\ref{algstep:in_wF}. This is unproblematic however, since $\initial_w(f)$ has fewer terms than $f$ for all $f\in F$, so that $\Trop(\initial_w(f))$ will be simpler than $\Trop(f)$. In fact, generically $\initial_w(f)$ will be a binomial and $\Trop(\initial_w(f))$ a linear space.

\begin{algorithm}[Testing for defects, weak genericity] \label{alg:weakGenericity} \
  \begin{algorithmic}[1]
    \REQUIRE{$(F,\lambda)$, where
      \begin{enumerate}[leftmargin=3mm]
      \item $F\subseteq K[x^{\pm 1}]$, a finite generating set of a $d$-dimensional prime ideal $I\subseteq K[x^{\pm 1}]$, and assume w.l.o.g. that
        \begin{equation}
          \label{eq:projectionWeak}\tag{$\ast$}
          \pi(\Trop(I)) = \RR^d,
        \end{equation}
        where $\pi:\RR^n\rightarrow\RR^d$ denotes the projection onto the first $d$ coordinates.
      \item $\lambda\in (K^\ast)^d$, describing an affine subspace $H:=\Trop(\{x_i-\lambda_i\mid i=1,\dots,d\})\subseteq\RR^n$ of complementary dimension $n-d$ such that the following \textit{weak genericity} assumption holds:
        \begin{equation}
          \tag{WG}\label{eq:weakGenericity}
          \Trop(I+\langle x_i-\lambda_i \mid i=1,\ldots,d\rangle) = \Trop(I)\cap_{\stable} H.
        \end{equation}
    \end{enumerate}
    }
    \ENSURE{$(b,\mathbf{w})$, such that
      \begin{enumerate}[leftmargin=*]
      \item if b=\texttt{true}, then $\mathbf{w}$ is a tropical defect,
      \item if b=\texttt{false}, then $\Trop(F)\cap_{\stable} H = \Trop(I)\cap_{\stable} H$. (In this case, $\mathbf{w} := 0$.)
      \end{enumerate}
    }
    \STATE Set $H:=\Trop(\{x_i-\lambda_i\mid i=1,\dots,d\})$ and $F':=F\cup\{x_i-\lambda_i\mid i=1,\dots,d \}$.
    \STATE Compute the tropical prevariety $\Trop(F')$. \AlgComment{0mm}{$\Trop(F')=\Trop(F)\cap H$}
    \STATE Initialize $\Delta:=\emptyset$. \AlgComment{0mm}{$\Delta$ will consist of tuples of weight vectors}
  \item[] \AlgComment{32.5mm}{first entry: weight vector in the stable intersection $\Trop(F)\cap_{\stable} H$}
  \item[] \AlgComment{32.5mm}{further entries: bookkeeping of the original cone in $\Trop(F)$}
    \FOR{$w\in\Trop(F')$ with $\dim C_w (F') = 0$}
    \STATE Compute $\Trop(\initial_w F)$. \label{algstep:in_wF}
    \IF{$\exists u\in \Trop(\initial_w F):\ \dim C_u(\initial_w F) > d$}
    \STATE Let $v_1,\ldots,v_k$ be a basis of $\Lin(C_u(\initial_w F))$.
    \RETURN{(\texttt{true}, $(w,u,v_1,\ldots,v_k)$).} \label{algstep:WGhighDim}
    \ENDIF
    \IF{$\exists u\in\Trop(\initial_w F)$ with $\dim(C_u(\initial_w F) + H)=n$}\label{algstep:DeltaBegin}
    \STATE Let $v_1,\ldots,v_d$ be a basis of $\Lin(C_u(\initial_w F))$.
    \STATE $\Delta := \Delta \cup \{(w,u,v_1,\dots,v_d)\}$.\label{algstep:DeltaEnd}
    \ENDIF
    \ENDFOR

    \STATE Compute $\Trop(I')$, where $I':= I+\langle x_i-\lambda_i\mid i=1,\dots,d\rangle$.
    \IF{$\exists (w,u,v_1,\dots,v_d) \in \Delta$ such that $w \notin \Trop(I')$}
    \RETURN{(\texttt{true}, $(w,u,v_1,\ldots,v_d)$).}\label{algstep:WGlowDim}
    \ELSE
    \RETURN{(\texttt{false}, $0$).}\label{algstep:WGNoDefects}
    \ENDIF
  \end{algorithmic}
\end{algorithm}

\begin{proof}[Correctness of Algorithm~\ref{alg:weakGenericity}]
  Suppose the algorithm terminates at Line~\ref{algstep:WGhighDim}. By Lemma~\ref{lem:nestedWeights}, there exists $\delta>0$ such that $D:=\{w+\varepsilon u+\varepsilon^2 v_1+\dots+\varepsilon^{k+1} v_k\mid 0<\varepsilon<\delta \}\subseteq \Trop(F)$. Because any infinite subset of $D$ has affine span $w+\Lin(C_u(\initial_w F))$ of dimension $k>d=\dim \Trop(I)$, any polyhedron on $\Trop(I)$ will have a finite intersection with $D$. In particular, this implies that $w+\varepsilon u+\varepsilon^2 v_1+\dots+\varepsilon^{k+1} v_k\notin\Trop(I)$ for $\varepsilon>0$ sufficiently small.

  Suppose the algorithm terminates at Line~\ref{algstep:WGlowDim}. Again, by Lemma~\ref{lem:nestedWeights}, there exists $\delta>0$ such that $D:=\{w+\varepsilon u+\varepsilon^2 v_1+\dots+\varepsilon^{d+1} v_d\mid 0<\varepsilon<\delta \}\subseteq \Trop(F)$.
  Any infinite subset of $D$ has affine span  $w+\Lin(C_u(\initial_w F))$, which intersects $H$ stably. We have $w\notin \Trop(I')=\Trop(I)\cap_{\stable}H$ by Assumption~\eqref{eq:weakGenericity}, so any polyhedron on $\Trop(I)$ around $w$ can only have a finite intersection with $D$.
  In particular, this implies that $w+\varepsilon u+\varepsilon^2 v_1+\dots+\varepsilon^{k+1} v_k\notin\Trop(I)$ for $\varepsilon>0$ sufficiently small.

  Finally, suppose the algorithm terminates at Line~\ref{algstep:WGNoDefects}. Since $\Trop(F)\supseteq\Trop(I)$, we always have $\Trop(F)\cap_{\stable} H\supseteq \Trop(I)\cap_{\stable} H$. For the converse, assume there exists a weight $w\in\Trop(F)\cap_{\stable} H \setminus \Trop(I)\cap_{\stable} H$. Let $C_{u}(F)\subseteq\Trop(F)$ be a Gr\"obner polyhedron of the prevariety with $w\in C_{u}(F)\cap H$ and $\dim(C_{u}(F) + H)=n$, which necessarily implies $\dim C_{u}(F)\geq d$. If $\dim C_{u}(F)>d$, then $\dim C_u(\initial_w(F))>d$
  and we would have terminated at Line~\ref{algstep:WGhighDim}. If $\dim C_{u}(F)=d$, then $w$ appears as the first entry of some tuple in $\Delta$ by Lemma~\ref{lem:nestedWeights} and Lines~\ref{algstep:DeltaBegin} to~\ref{algstep:DeltaEnd}, hence we would have terminated at Line~\ref{algstep:WGlowDim}, as $\Trop(I')=\Trop(I)\cap_{\stable} H$ by Assumption~\eqref{eq:weakGenericity}.
\end{proof}

\begin{remark}[Weak genericity]
  If Algorithm~\ref{alg:weakGenericity} terminates at Line~\ref{algstep:WGhighDim}, then the output is correct even if the input did not satisfy the weak genericity assumption \eqref{eq:weakGenericity}, since a polyhedron in $\Trop(F)$ of too large dimension was found.
  On the other hand, the correctness of a tropical defect output at Step~\ref{algstep:WGlowDim} does depend on the assumption \eqref{eq:weakGenericity} on the input. In order to certify the correctness of the output regardless of the validity of \eqref{eq:weakGenericity},
  one needs to check that there is no sufficiently small $\varepsilon > 0$ such that $w+\varepsilon u + \varepsilon^2 v_1+\ldots+\varepsilon^{d+1} v_d \in \Trop I$. If $I$ is homogeneous, this can by \cref{lem:nestedWeights} be achieved by certifying that the iterated initial ideal $\initial_{v_d} \cdots \initial_{v_1} \initial_u \initial_w I$ is the entire Laurent polynomial ring $\mathfrak{K}[x^{\pm 1}]$.
\end{remark}

\begin{example}
  Consider the generating set from Example~\ref{ex:strongGenericity} (see also Figure~\ref{fig:exampleStrongGenericity}):
  \[ I:= \langle \underbrace{(x+1)(y+1),(x-1)(y+1)}_{=:F} \rangle \subseteq \CC[x^{\pm1},y^{\pm1}]. \]
  Unlike before, Algorithm~\ref{alg:weakGenericity} will be unable to find a tropical defect around $H_v$ even for $v=0$, always terminating at Line~\ref{algstep:WGNoDefects}. This is because without condition~\eqref{eq:strongGenericity} $H_0$ need not have a zero-dimensional intersection with $\Trop(I)$, so that its positive-dimensional intersection with $\Trop(F)$ need not arise from a tropical defect.

  However Algorithm~\ref{alg:weakGenericity} will still find a tropical defect for $L_v$ for $v\neq 0$, in which case it terminates at Line~\ref{algstep:WGlowDim}.
\end{example}

\begin{remark}[Strong genericity vs.\ weak genericity from a practical point of view]\label{rem:strongVsWeak}
  Theoretically, it is always possible to find tropical defects for generating sets which are not tropical bases using Algorithm~\ref{alg:strongGenericity} with the right choice of an affine subspace. In practice, however, it is much more reasonable to use Algorithm~\ref{alg:weakGenericity} instead.
  This is because generic $v\in\RR^d$ for Algorithm~\ref{alg:strongGenericity} usually entail high exponents in the polynomial computations, whereas generic $\lambda\in (K^\ast)^d$ for Algorithm~\ref{alg:weakGenericity} only entail big coefficients, and most computer-algebra software systems such as \textsc{Macaulay2} or \textsc{Singular} are better equipped to deal with the latter. For instance, our \textsc{Singular} experiments using Algorithm~\ref{alg:strongGenericity} regularly failed due to exponent overflows, since exponents in \textsc{Singular} are stored in the \Cplusplus{} type \texttt{signed short} (bounded by $2^{15}$ for most CPU architectures), while coefficients are stored with arbitrary precision.
\end{remark}

\begin{remark}[Comparison with existing techniques]
  As hinted in the introduction, tropical basis verification is a problem that has been studied by many people. However, the only software currently capable of this task is \textsc{gfan} \cite{gfan}, which for example has been used to prove that the $4\!\times\! 4$-minors of a $5\!\times\! n$ matrix form a tropical basis \cite{CJR11}. Its command \texttt{gfan\_tropicalbasis} computes a tropical basis of a tropical curve, and its command \texttt{gfan\_tropicalintersection} for computing tropical prevarieties $\Trop(F)$ has an optional argument \texttt{--tropicalbasistest} to test whether $\Trop(F)$ equals the tropical variety $\Trop(I)$. Compared to the algorithms in \textsc{gfan}, our techniques have the following disadvantages and advantages.

  Since our algorithms revolve around finding tropical defects, they are incapable to verify that a generating set is a tropical basis. As we only search around random hyperplanes of complementary dimension, we are also blind to lower-dimensional defects, i.e. if $\dim(\Trop(I)\setminus\Trop(F))<\dim(\Trop(I))=:d$ then the probability for a random affine hyperplane of codimension $d$ to intersect $\Trop(I)\setminus\Trop(F)$ is zero. One example where our algorithms failed to return a definite answer is \cite[Conjecture 4.8]{Rincon12}.

  In return, our algorithms avoid the computation of both $\Trop(F)$ and $\Trop(I)$. Instead of $\Trop(F)=\bigcap_{f\in F}\Trop(f)$, we compute $\Trop(F')=\bigcap_{f\in F} (\Trop(f)\cap H)$. This is faster, since $\Trop(f)\cap H$ is covered by fewer polyhedra compared to $\Trop(f)$. Moreover, instead of $\Trop(I)$ we compute $\Trop(I')$, where $I':= I+\langle x_i\!-\!\lambda_i\mid i=1,\dots,d\rangle$. This is easier since $I'$ is zero-dimensional whereas $I$ is not. Additionally, $\Trop(I')$ consists of up to $\deg(I)$ many points while $\Trop(I)$ is generally covered by many more polyhedra.
\end{remark}

\section{Application: Cox rings of cubic surfaces}\label{sec:DelPezzo}
Cox rings are global invariants of important classes of algebraic varieties. For example, they carry essential information about all morphisms to projective spaces and play a central role in the theory of universal torsors, see \cite{ADHL15} for further details.
In this section, we address \cite[Conjecture~5.3]{RSS16} on Cox rings of smooth cubic surfaces, disproving it with a tropical defect.

\begin{definition}
  Consider six points $p_1, \ldots, p_6 \in \P_\CC^2$ in general position in the complex projective plane. Up to change of coordinates, we may assume that
  \[ p_i=(1:d_i:d_i^3) \quad \text{for some } d_i\in\CC, \]
where $d_i$ satisfy certain genericity conditions, see \cite[§6]{RSS14}. Blowing up $\P_\CC^2$ in these points results in a smooth cubic surface $X := \Bl_{p_1, \ldots, p_6} \P_\CC^2$. The geometry of this surface is captured by its \emph{Cox ring}
  \[\Cox(X) := \bigoplus_{(a_0,\ldots,a_6) \in \ZZ^7} H^0(X, \O_X(a_0 E_0 + a_1 E_1 + \ldots + a_6 E_6)),\]
  where
  \begin{itemize}[leftmargin=*,label={\bluetriangleright}]
  \item $E_1, \ldots, E_6 \subseteq X$ are the exceptional divisors over the points $p_1, \ldots, p_6 \in \P_\CC^2$,
  \item $E_0 \subseteq X$ is the preimage of a line in $\P_\CC^2$ not containing $p_1, \ldots, p_6$, and
  \item $H^0(X, \O_X(a_0 E_0 + a_1 E_1 + \ldots + a_6 E_6)) \subseteq K(X)$ are the rational functions on $X$ which vanish along each $E_i$ with multiplicity at least $-a_i$ (vanishing with negative multiplicity meaning poles of positive order).
  \end{itemize}
\end{definition}

For a smooth cubic surface $X$, the Cox ring $\Cox(X)$ is a finitely generated integral domain with a natural set of $27$~generators which are the rational functions on $X$ establishing the linear equivalence of each of the $27$~lines on the cubic surface $X$ to a divisor of form $\sum_i a_i E_i \in \Div(X)$, see \cite[Theorem~3.2]{BP04}.

\begin{proposition}[{\cite[Proposition~2.2]{RSS16}}]\label{def:coxRing}
  Let $d_1,\ldots,d_6\in\CC$ and $X$ be the cubic surface that is the blowup of $(1:d_i:d_i^3)\in\P_\CC^2$. Then
  \[ \Cox(X)\cong \bigslant{\CC[E_1,\dots,E_6,F_{12},F_{13},\dots,F_{56},G_1,\dots,G_6]}{I_X}, \]
  where, up to saturation at the product of all variables, $I_X$ is generated by the following $10$~trinomials and their $260$~translates under the action of the Weyl group of type~$\mathbf{E}_6$:
  \begin{small}\allowdisplaybreaks
    \begin{align*}
      &(d_3{-}d_4)(d_1{+}d_3{+}d_4)E_2F_{12} - (d_2{-}d_4)(d_1{+}d_2{+}d_4)E_3F_{13} + (d_2{-}d_3)(d_1{+}d_2{+}d_3)E_4F_{14} ,\\
      &(d_3{-}d_5)(d_1{+}d_3{+}d_5)E_2F_{12} - (d_2{-}d_5)(d_1{+}d_2{+}d_5)E_3F_{13} + (d_2{-}d_3)(d_1{+}d_2{+}d_3)E_5F_{15} ,\\
      &(d_3{-}d_6)(d_1{+}d_3{+}d_6)E_2F_{12} - (d_2{-}d_6)(d_1{+}d_2{+}d_6)E_3F_{13} + (d_2{-}d_3)(d_1{+}d_2{+}d_3)E_6F_{16} , \\
      &(d_4{-}d_5)(d_1{+}d_4{+}d_5)E_2F_{12} - (d_2{-}d_5)(d_1{+}d_2{+}d_5)E_4F_{14} + (d_2{-}d_4)(d_1{+}d_2{+}d_4)E_5F_{15} , \\
      &(d_4{-}d_6)(d_1{+}d_4{+}d_6)E_2F_{12} - (d_2{-}d_6)(d_1{+}d_2{+}d_6)E_4F_{14} + (d_2{-}d_4)(d_1{+}d_2{+}d_4)E_6F_{16} , \\
      &(d_5{-}d_6)(d_1{+}d_5{+}d_6)E_2F_{12} - (d_2{-}d_6)(d_1{+}d_2{+}d_6)E_5F_{15} + (d_2{-}d_5)(d_1{+}d_2{+}d_5)E_6F_{16} , \\
      &(d_4{-}d_5)(d_1{+}d_4{+}d_5)E_3F_{13} - (d_3{-}d_5)(d_1{+}d_3{+}d_5)E_4F_{14} + (d_3{-}d_4)(d_1{+}d_3{+}d_4)E_5F_{15} , \\
    &(d_4{-}d_6)(d_1{+}d_4{+}d_6)E_3F_{13} - (d_3{-}d_6)(d_1{+}d_3{+}d_6)E_4F_{14} + (d_3{-}d_4)(d_1{+}d_3{+}d_4)E_6F_{16} ,\\
      &(d_5{-}d_6)(d_1{+}d_5{+}d_6)E_3F_{13} - (d_3{-}d_6)(d_1{+}d_3{+}d_6)E_5F_{15} + (d_3{-}d_5)(d_1{+}d_3{+}d_5)E_6F_{16} ,\\
      &(d_5{-}d_6)(d_1{+}d_5{+}d_6)E_4F_{14} - (d_4{-}d_6)(d_1{+}d_4{+}d_6)E_5F_{15} + (d_4{-}d_5)(d_1{+}d_4{+}d_5)E_6F_{16}.
    \end{align*}
  \end{small}\vspace{-1em}

  \noindent Here,
  \begin{itemize}[leftmargin=*,label={\bluetriangleright}]
  \item $E_i$ represents the exceptional divisor over the point $p_i$,
  \item $F_{ij}$ represents the strict transform of the line through $p_i$ and $p_j$,
   \item $G_i$ represents the strict transform of the conic through $\{p_1,\ldots,p_6\} \setminus \{p_i\}$.
  \end{itemize}
\end{proposition}

The following theorem answers~\cite[Conjecture 5.3]{RSS16} negatively:

\begin{theorem}\label{thm:cox}
  For generic $d_1,\ldots,d_6 \in \CC$, the $270$ trinomial generators of $I_X$ described in \cref{def:coxRing} are not a tropical basis.
\end{theorem}
\begin{proof}
  Fix the following ordered set of variables:
  \begin{align*}
    S := \{&E_1,E_2,E_3,E_4,E_5,E_6,F_{12},F_{13},F_{14},F_{15},F_{16},F_{23},F_{24},F_{25},F_{26},\\
    &F_{34},F_{35},F_{36},F_{45},F_{46},F_{56},G_1,G_2,G_3,G_4,G_5,G_6\}.
  \end{align*}
  Let $I_X$ be the ideal in the polynomial ring $\CC(d_1,\dots,d_6)[S]$ generated by the $270$ trinomials described in \cref{def:coxRing}, and consider the weight vector
  \[ w:=(2,1,0,1,1,1,0,2,0,0,0,1,0,0,0,1,1,1,0,0,0,0,0,0,0,0,0)\in\RR^{S}. \]
  One can verify that $w$ is a tropical defect, i.e., $w$ lies in the tropical prevariety, since $\initial_w(f)$ is at least binomial for each trinomial generator $f$, and outside the tropical variety, since $\initial_w(I_X)$ contains the monomial $E_6F_{56}G_6$.
\end{proof}

\begin{remark}
  The statements in the proof of Theorem~\ref{thm:gaussoids} can be easily verified using a computer algebra system such as \textsc{Singular}. The following script is available on \href{https://software.mis.mpg.de}{software.mis.mpg.de}, and the following shortened transcript was produced using \textsc{Singular}'s online interface (version 4.1.1) available at \href{https://www.singular.uni-kl.de/tryonline}{singular.uni-kl.de/tryonline}:

\begin{lstlisting}[basicstyle=\scriptsize\ttfamily,
  keywordstyle=\color{red}\ttfamily,
  stringstyle=\color{blue}\ttfamily,
  morekeywords={LIB, int, intvec, ring, ideal, poly},
  escapechar=@]
> LIB "tropicalBasis.lib"; @\LstComment{0mm}{initializes necessary libraries and helper functions}@
> intvec wMin = 2,1,0,1,1,1,0,2,0,0,0,1,0,0,0,1,1,1,0,0,0,0,0,0,0,0,0;
                                 @\LstComment{0mm}{wMin is in min-convention}@
> intvec wMax = -wMin;           @\LstComment{0mm}{\textsc{Singular} uses max-convention}@
> intvec allOnes = onesVector(size(wMax));
> ring r = (0,d1,d2,d3,d4,d5,d6),(E1,E2,E3,E4,E5,E6,
.   F12,F13,F14,F15,F16,F23,F24,F25,F26,F34,F35,F36,F45,F46,F56,
.   G1,G2,G3,G4,G5,G6),(a(allOnes),a(wMax),lp);
                          @\LstComment{0mm}{prepending \texttt{allOnes} makes no difference mathematically}@
                          @\LstComment{0mm}{as the ideal is homogeneous,}@
                          @\LstComment{0mm}{but it helps computationally}@
> ideal F =               @\LstComment{0mm}{\textsc{Singular} ideals are lists of polynomials}@
.  (d3-d4)*(d1+d3+d4)*E2*F12+(d2-d4)*(d1+d2+d4)*E3*F13
.    -(d2-d3)*(d1+d2+d3)*E4*F14,
@\tiny$\hspace{1.75pt}\vdots$@   [...]
.  -(d5-d6)*(d1+d3+d4)*F24*G4+(d4-d6)*(d1+d3+d5)*F25*G5
.    -(d4-d5)*(d1+d3+d6)*F26*G6;
> ideal inF = initial(F,wMax); @\LstComment{0mm}{initial forms of the elements in F}@
                               @\LstComment{0mm}{all are at least binomial, hence $\text{wMax}\in\Trop(F)$}@
> ideal IX = groebner(F);
> ideal inIX = initial(IX,wMax);@\LstComment{0mm}{initial forms of Gr\"obner basis elements}@
                                @\LstComment{0mm}{this is a Gr\"obner basis of $\initial_{\text{wMax}}(I_X)$}@
> NF(E6*F56*G6,inIX);           @\LstComment{0mm}{normal form is $0$ hence $E_6*F_{56}*G_6\in\initial_{\text{wMax}}(I_X)$}@
0
\end{lstlisting}
\end{remark}

\section{Application: Realizability of valuated gaussoids}\label{sec:Gaussoids}
Gaussoids are combinatorial structures introduced by Ln\v eni\v cka and Mat\'u\v s \cite{LM07} that encode conditional independence relations among Gaussian random variables. Reminiscent of the study of matroids, Boege, D'Al\`i, Kahle and Sturmfels \cite{BDKS17} introduced the notions of oriented and valuated gaussoids.
In this section, we address the question whether all valuated gaussoids on four elements are realizable, disproving it with a tropical defect. This was initially conjectured in the first version of \cite{BDKS17}, as found on arXiv. The published version has since been updated with our \cref{thm:gaussoids}.

\begin{definition}[{\cite[\S1]{BDKS17}}] \label{def:gaussoids}
  Fix $n\in\NN$. Consider the Laurent polynomial ring
  \[R_n := \CC\big[p_{I}^{\pm 1}\mid I\subseteq [n]\big]\big[a_{\{i,j\}|K}^{\pm 1}\mid i,j \in [n] \text{ distinct}, K\subseteq [n]\setminus\{i,j\}\big], \]
  in which we abbreviate $a_{\{i,j\}|K}$ to $a_{ij|K}$, and the ideal $T_n$ generated by the following
  $2^{n-2}\binom{n}{2}$ square trinomials and the following $12\cdot 2^{n-3}\binom{n}{3}$ edge trinomials:
  \begin{align*}
    & a_{ij|K}^2 - p_{K \cup \{i\}}\:p_{K \cup \{j\}} + p_{K \cup \{i,j\}}\:p_K \quad \text{for } i,j \in [n] \text{ distinct, } K\subseteq [n]\setminus \{i,j\}, \\[2mm]
    &p_{L \cup \{k\}}\: a_{ij|L \setminus \{i,j\}} - p_{L}\: a_{ij|L \cup \{k\} \setminus \{i,j\}} - a_{k i |L \setminus \{i\}}\: a_{k j|L\setminus \{j\}}\\[1mm]
    & \hspace{65mm}\text{for } i,j,k \in [n] \text{ distinct, } L\subseteq [n]\setminus \{k\}.
  \end{align*}

  A \emph{valuated gaussoid} is a point in the tropical prevariety defined by the square and edge trinomials. It is called \emph{realizable} if it lies in the tropical variety $\Trop(T_n)$.
\end{definition}

\begin{remark}
  The variables of the ring $R$ correspond to the principal and
  almost-principal minors of a symmetric $n \times n$-matrix (i.e., determinants
  of square submatrices whose row- and column index sets differ by at most one
  index). The ideal $T_n$ corresponds to the polynomial relations among these minors for
  symmetric matrices with non-zero principal minors by
  \cite[Proposition~6.2]{BDKS17}.
\end{remark}

The following theorem negatively answers Conjecture~8.4 in the first arXiv-version of \cite{BDKS17}, and is now Theorem~8.4 in the final published version of \cite{BDKS17}:

\begin{theorem}\label{thm:gaussoids}
  Not all valuated gaussoids on four elements are realizable, i.e., the square and edge trinomials in Definition~\ref{def:gaussoids} are not a tropical basis of $T_4$.
\end{theorem}
\begin{proof}
  Consider the following ordered set $S$ of the variables of $R_4$ and weight vector $w\in\RR^S$:
  \allowdisplaybreaks
  \begin{align*}
    S\!:=\!\{&p_{\emptyset},p_{1},p_{12},p_{123},p_{1234},p_{124},p_{13},p_{134},p_{14},p_{2},p_{23},p_{234},p_{24},p_{3},p_{34},p_{4},\\
             &a_{12},a_{12|3},a_{12|34},a_{12|4},a_{13},a_{13|2},a_{13|24},a_{13|4},a_{14},a_{14|2},a_{14|23},a_{14|3},\\
             &a_{23},a_{23|1},a_{23|14},a_{23|4},a_{24},a_{24|1},a_{24|13},a_{24|3},a_{34},a_{34|1},a_{34|12},a_{34|2}\}\\[1mm]
    w\!:=\!(&14,10,6,0,6,8,8,2,8,6,6,2,8,8,8,8,8,4,2,10,9,3,5,5,9,11,\\ %
             &1,5,7,5,5,5,7,7,1,5,8,6,4,4)\in\RR^S.
  \end{align*}
  One can check that $w$ is a tropical defect, i.e., $w$ lies in the tropical prevariety, since $\initial_w(f)$ is at least binomial for all square and edge trinomials, and outside the tropical variety, since $\initial_w(T_4)$ contains the monomial $a_{23}a_{23|1}$.
\end{proof}

\begin{remark}
  The statements in the proof of Theorem~\ref{thm:gaussoids} can be easily verified using a computer algebra system such as \textsc{Singular}. The following script is available on \href{https://software.mis.mpg.de}{software.mis.mpg.de}, and the following shortened transcript was produced using \textsc{Singular}'s online interface (version 4.1.1) available at \href{https://www.singular.uni-kl.de/tryonline}{singular.uni-kl.de/tryonline}:

\begin{lstlisting}[basicstyle=\scriptsize\ttfamily,
  keywordstyle=\color{red}\ttfamily,
  stringstyle=\color{blue}\ttfamily,
  morekeywords={LIB, int, intvec, ring, ideal, poly},
  escapechar=@]
> LIB "tropicalBasis.lib"; @\LstComment{0mm}{initializes necessary libraries and helper functions}@
> intvec wMin = 14,10,6,0,6,8,8,2,8,6,6,2,8,8,8,8,8,4,2,10,9,3,5,5,9,11,
.   1,5,7,5,5,5,7,7,1,5,8,6,4,4; @\LstComment{0mm}{wMin is in min-convention}@
> intvec wMax = -wMin;           @\LstComment{0mm}{\textsc{Singular} uses max-convention}@
> intvec allOnes = onesVector(size(wMax));
> ring r = 0,(p,p1,p12,p123,p1234,p124,p13,p134,p14,p2,p23,p234,p24,p3,p34,p4,
.   a12,a12_3,a12_34,a12_4,a13,a13_2,a13_24,a13_4,a14,a14_2,a14_23,a14_3,
.   a23,a23_1,a23_14,a23_4,a24,a24_1,a24_13,a24_3,a34,a34_1,a34_12,a34_2),
.   (a(allOnes),a(wMax),lp); @\LstComment{0mm}{prepending \texttt{allOnes} makes no difference mathematically}@
                             @\LstComment{0mm}{as the ideal is homogeneous,}@
                             @\LstComment{0mm}{but it helps computationally}@
> ideal F =                  @\LstComment{0mm}{\textsc{Singular} ideals are lists of polynomials}@
.   a34_12*a13_24+p124*a14_23-a14_2*p1234,
@\tiny$\hspace{1.75pt}\vdots$@   [...]
.   -p1*p2+a12^2+p*p12;
> ideal inF = initial(F,wMax); @\LstComment{0mm}{initial forms of the elements in F}@
                               @\LstComment{0mm}{all are at least binomial, hence $\text{wMax}\in\Trop(F)$}@
> ideal I = groebner(F);
> ideal inI = initial(I,wMax); @\LstComment{0mm}{initial forms of all elements in the Gr\"obner basis}@
                               @\LstComment{0mm}{this is a Gr\"obner basis of $\initial_{\text{wMax}}(I)$}@
> NF(a23*a23_1,inI);           @\LstComment{0mm}{normal form is $0$ hence $a_{23}a_{23|1}\in\initial_{\text{wMax}}(I)$}@
0
\end{lstlisting}
\end{remark}

\begin{remark}[sampling affine subspaces for tropical defects]
  The tropical defects in Theorems \ref{thm:cox} and \ref{thm:gaussoids} were found by repeatedly running Algorithm~\ref{alg:weakGenericity} on random affine subspaces $H\subseteq\RR^n$. In the sampling of the affine subspaces, a situation which we tried to avoid are two subspaces intersecting the tropical variety in exactly the same Gr\"obner polyhedra. In the following, we describe our sampling approach which we based on this thought.

  Even though we were unable to compute the tropical variety $\Trop(I)$ or the tropical prevariety $\Trop(F)$ in both problems, we were able to compute
  \begin{enumerate}[leftmargin=*]
  \item a Gr\"obner basis of $I$ with respect to a graded reverse lexicographical ordering,
  \item for selected finite fields $\FF$ and $d+1:=\dim(I)+1$ variables $x_{i_0},\dots,x_{i_d}$, the generator $\overline g\in \FF [x_{i_0},\dots,x_{i_d}]$ of the principal elimination ideal $(I\otimes_{\ZZ}\FF) \cap \FF [x_{i_0},\dots,x_{i_d}]$.
  \end{enumerate}
  In other words, (2) allowed for educated guesses for generators $g$ of principal elimination ideals $I \cap K[x_{i_0},\dots,x_{i_d}]$, while (1) allowed for tests whether the guesses were correct.
  Thus, we were able to compute tropical hypersurfaces $\Trop(g)\subseteq\RR^{d+1}$ which are the images of $\Trop(I)$ under selected orthogonal projections $\pi:\RR^n\rightarrow\RR^{d+1}$.

  For each projection, we then constructed affine lines $L_1,\ldots,L_k\subseteq\RR^{d+1}$ such that each maximal polyhedron of $\Trop(g)$ intersects at least one line. Their preimages $\pi^{-1}L_1,\ldots,\pi^{-1}L_k$ are then $d$-codimensional affine subspaces which were our samples for $H$.
\end{remark}

\renewcommand{\emph}[1]{\textit{#1}}

\end{document}